\documentclass[a4paper,11pt]{amsart}
\usepackage{amsmath,amssymb,amsthm,mathtools}
\usepackage{hyperref,enumitem,xcolor}
\usepackage{cite}

\DeclareMathOperator{\sgn}{sgn}
\DeclareMathOperator{\supp}{supp}
\DeclareMathOperator{\ext}{ext}

\newtheorem{theorem}{Theorem}
\newtheorem{lemma}[theorem]{Lemma}
\newtheorem{proposition}[theorem]{Proposition}

\theoremstyle{remark}

\theoremstyle{definition}

\title[Plasticity of the unit ball of the real Banach space $\ell_\infty$]{Plasticity of the unit ball of the real Banach space $\ell_\infty$}

\author{Rainis Haller}
\address{Institute of Mathematics and Statistics, University of Tartu, Estonia}
\email{rainis.haller@ut.ee}

\author{Nikita Leo}
\address{Institute of Mathematics and Statistics, University of Tartu, Estonia}
\email{nikita.leo@ut.ee}

\thanks{This work was supported by the Estonian Research Council grant
PRG1901.}
\date{}

\begin{document}
\begin{abstract}
We prove that the closed unit ball of the real Banach space $\ell_\infty$ is plastic, that is, every non-expansive bijection from the unit ball onto itself is an isometry. 
The main step is to show that every non-expansive bijection of this ball maps extreme points to extreme points. This is done by using elementary coverings of the unit ball by balls of radius one. The conclusion then follows from a theorem by Fakhoury.
The same argument also shows that, for arbitrary $\Gamma$, every non-expansive bijection of $B_{\ell_\infty(\Gamma)}$ preserves extreme points.
\end{abstract}

\maketitle
\section{Introduction}
A mapping between metric spaces is called \emph{non-expansive} if it is $1$-Lipschitz. Following Naimpally, Piotrowski, and Wingler \cite{MR2178720}, a metric space $M$ is called \emph{plastic} if every non-expansive bijection $F\colon M\to M$ is an isometry.

A natural question asks whether the closed unit ball $B_X$ of every real Banach space $X$ is plastic. This problem was explicitly posed and studied by Cascales, Kadets, Orihuela, and Wingler \cite{MR3534519}, who proved that the unit ball is plastic in every strictly convex Banach space. 
Since then, several further positive results have been obtained, including the cases of $\ell_1$ and $c$; see, for instance, \cite{MR3906312, MR4633724, MR4442730, zbMATH06742079, MR3772659, zbMATH07428665}.

We shall use a result of Fakhoury \cite{MR4633724}, which says that for $B_{\ell_\infty}$ a non-expansive bijection preserving extreme points is an isometry.


Levchenko and Zavarzina \cite{arXiv:2605.06622} recently proved plasticity results for the unit spheres of $\ell_1$, $\ell_\infty$, and $c$. Their result for $S_{\ell_\infty}$ does not imply the result below for $B_{\ell_\infty}$.

Our main result is the following.

\begin{theorem}\label{thm: ell-infty-plastic}
    Let $B$ be the closed unit ball of the real Banach space $\ell_\infty$. Then every non-expansive bijection $F\colon B\to B$ is an isometry. Equivalently, $B_{\ell_\infty}$ is plastic.
\end{theorem}

The proof is elementary up to the final application of Fakhoury's result. We use the following simple observation: if $0\neq a,b\in B$ and $B\subset B(a,1)\cup B(b,1)$, then $a$ and $b$ must be non-zero opposite multiples of the same coordinate vector. Applying this to the images of the coverings $B\subset B(e_n,1)\cup B(-e_n,1)$ shows that the intervals $[-e_n,e_n]$ are sent into coordinate lines, up to signs and a permutation of the coordinates. The resulting mapping on coordinates is injective, and a support argument then shows that it is in fact a bijection. After composing with the inverse signed permutation, the resulting map preserves the signs of coordinates. This is enough to show that all extreme points of $B$ are fixed by the normalised map, and hence that the original map sends extreme points to extreme points.

At the end, we add a short remark on $\ell_\infty(\Gamma)$. The same coordinate part of the proof works for arbitrary $\Gamma$ and gives preservation of extreme points. We do not claim plasticity of $B_{\ell_\infty(\Gamma)}$ for arbitrary $\Gamma$.

\section{Proof of the main theorem}


Let $F\colon B\to B$ be a non-expansive bijection. Then $F(0)=0$. Indeed, if $a=F(0)$, then, by surjectivity, $B\subset B(a,1)$. If $a\neq 0$, choose $k\in\mathbb N$ such that $a_k\neq 0$, and choose $x\in B$ with $x_k=-\sgn(a_k)$. Then $\|a-x\|\geq 1+|a_k|>1$, a contradiction. Hence $a=0$.

\begin{lemma}\label{lem:1}
    Let $a_i\in B$ and $r_i>0$ for every $i\in I$. If $B\subset \bigcup_{i\in I}B(a_i,r_i)$, then $B\subset \bigcup_{i\in I}B(F(a_i), r_i)$.
\end{lemma}

\begin{proof}
    Fix $y\in B$. Let $x\in B$ be such that $F(x)=y$, and let $i\in I$ be such that $x\in B(a_i,r_i)$. Then 
    \begin{align*}
        \|y-F(a_i)\|=\|F(x)-F(a_i)\|\leq \|x-a_i\|\leq r_i.
    \end{align*}
    Thus $y\in B(F(a_i),r_i)$.
\end{proof}

\begin{lemma}\label{lem:2}
    Let $0\neq a,b\in B$ be such that $B\subset B(a,1)\cup B(b,1)$. Then there exist $i\in \mathbb N$ and $\alpha,\beta>0$ such that either $a=\alpha e_i$ and $b=-\beta e_i$, or $a=-\alpha e_i$ and $b=\beta e_i$.
\end{lemma}

\begin{proof}
    Choose $i$ such that $a_i\neq 0$. Suppose that for some $j\neq i$ one has $b_j\neq 0$. Let $x\in B$ be such that $x_i=-\sgn(a_i)$, $x_j=-\sgn(b_j)$, and all other coordinates are $0$. Then $|x_i-a_i|=1+|a_i|>1$, so $x\notin B(a,1)$. Also $|x_j-b_j|=1+|b_j|>1$, so $x\notin B(b,1)$, a contradiction. Therefore $b_j=0$ for every $j\neq i$. The same argument with the roles of $a$ and $b$ interchanged gives $a_j=0$ for every $j\neq i$. Thus $a=a_i e_i$ and $b=b_i e_i$. If $a_i$ and $b_i$ had the same sign, then $-\sgn(a_i)e_i\notin B(a,1)\cup B(b,1)$, a contradiction. Hence, $a_i$ and $b_i$ have different signs.
\end{proof}

Fix $n\in\mathbb N$. Since $B\subset B(e_n,1)\cup B(-e_n,1)$, Lemma~\ref{lem:1} gives $B\subset B(F(e_n),1)\cup B(F(-e_n),1)$. Since $F(0)=0$ and $F$ is injective, $F(e_n)\neq 0$ and $F(-e_n)\neq 0$. By Lemma~\ref{lem:2}, there exist $\sigma(n)\in\mathbb N$, $\theta_n\in\{-1,1\}$, and $\alpha_n,\beta_n>0$ such that $F(e_n)=\theta_n\alpha_ne_{\sigma(n)}$ and $F(-e_n)=-\theta_n\beta_ne_{\sigma(n)}$.

If $0<t\leq 1$, then $B\subset B(te_n,1)\cup B(-e_n,1)$. Since $F(te_n)\neq 0$, Lemma~\ref{lem:2} implies that $F(te_n)=\theta_n \gamma_n(t)e_{\sigma(n)}$ for some $\gamma_n(t)>0$. If $-1\leq t<0$, then $B\subset B(e_n,1)\cup B(te_n,1)$. Since $F(te_n)\neq 0$, Lemma~\ref{lem:2} implies that $F(te_n)=\theta_n \gamma_n(t)e_{\sigma(n)}$ for some $\gamma_n(t)<0$. Set also $\gamma_n(0)=0$.
Since $F(0)=0$, we have $F(0e_n)=\theta_n\gamma_n(0)e_{\sigma(n)}$.

The function $t\mapsto F(te_n)$ is continuous.
Therefore $F([-e_n,e_n])$ is a non-trivial interval in $\mathbb R e_{\sigma(n)}$ which contains $0$ as an interior point.

We next show that $\sigma$ is injective. Suppose that for some $m\neq n$ one has $\sigma(m)=\sigma(n)$. Then $F([-e_m,e_m])$ and $F([-e_n,e_n])$ are both non-trivial intervals in the same line $\mathbb Re_{\sigma(n)}$, and $0$ is an interior point of both. Thus, they have a common non-zero element, say $F(se_m)=F(te_n)$ for some $s,t\neq 0$. By injectivity of $F$, we have $se_m=te_n$, which is impossible since $m\neq n$. Hence $\sigma$ is injective.

Fix $x\in B$, and recall that $\supp x=\{n\in\mathbb N : x_n\neq 0\}$. We claim that 
\[
B\subset B(x,1)\cup\bigcup_{n\in\supp x}B(-\sgn(x_n)e_n,1).
\]
Indeed, let $y\in B$ and suppose that $y\notin \bigcup_{n\in\supp x}B(-\sgn(x_n)e_n,1)$.
Then, for every $n\in\supp x$, $x_n$ and $y_n$ have the same sign, and thus $|x_n-y_n|\leq 1$. If $n\notin \supp x$, then $x_n=0$, and so $|x_n-y_n|=|y_n|\leq 1$. Hence $y\in B(x,1)$.

By Lemma~\ref{lem:1}, $B\subset B(F(x),1)\cup\bigcup_{n\in\supp x}B(F(-\sgn (x_n) e_n),1)$. Let $y\in B$ be defined by $y_{\sigma(n)}=\theta_n\sgn (x_n)$ for every $n\in\supp x$, and let all remaining coordinates be $0$. We show that $y\notin B(F(-\sgn (x_n) e_n),1)$ for every $n\in\supp x$. If $x_n>0$, then $-\sgn (x_n)e_n=-e_n$ and $F(-e_n)=-\theta_n\beta_n e_{\sigma(n)}$, whence
\[
\bigl|y_{\sigma(n)}-(F(-\sgn (x_n) e_n))_{\sigma(n)}\bigr|=|\theta_n+\theta_n \beta_n|=1+\beta_n>1.
\]
If $x_n<0$, then $-\sgn (x_n)e_n=e_n$ and $F(e_n)=\theta_n\alpha_n e_{\sigma(n)}$, whence
\[
\bigl|y_{\sigma(n)}-(F(-\sgn(x_n) e_n))_{\sigma(n)}\bigr|=|-\theta_n-\theta_n \alpha_n|=1+\alpha_n>1.
\]
Therefore $y\in B(F(x),1)$.

If $x_n>0$, then $y_{\sigma(n)}=\theta_n$, and since $|y_{\sigma(n)}-(F(x))_{\sigma(n)}|\leq 1$, we get $\theta_n(F(x))_{\sigma(n)}\geq 0$. Similarly, if $x_n<0$, then $\theta_n(F(x))_{\sigma(n)}\leq 0$.

Let $k\notin \sigma(\supp x)$. Redefine $y$ in the coordinate $k$ by setting $y_k=1$. The modified point still avoids all the balls in the union over $n\in\supp x$. Hence the modified point belongs to $B(F(x),1)$, and so $|1-(F(x))_{k}|\leq 1$. Similarly, redefining $y$ in the coordinate $k$ by setting $y_k=-1$, we obtain $|-1-(F(x))_{k}|\leq 1$. Thus $(F(x))_k=0$. We have proved that 
\[
\supp(F(x))\subset \sigma(\supp x),
\]
and moreover, for every $n$, $x_n>0$ implies $\theta_n(F(x))_{\sigma(n)}\geq 0$, and $x_n<0$ implies $\theta_n(F(x))_{\sigma(n)}\leq 0$.

We now show that $\sigma$ is surjective. Suppose that some $k\notin \sigma(\mathbb N)$. Then $k\notin \sigma(\supp x)$ for every $x\in B$, and therefore $(F(x))_k=0$ for every $x\in B$. Since $F$ is surjective, there exists $x\in B$ such that $F(x)=e_k$. But then $1=(e_k)_k=(F(x))_k=0$, a contradiction. Hence $\sigma$ is a bijection.

Consider the surjective linear isometry $A\colon \ell_\infty \to\ell_\infty$ defined by $(Ax)_{\sigma(n)}=\theta_n x_n$. Set $G=A^{-1}\circ F$. Then $G\colon B\to B$ is a non-expansive bijection. By what we have already proved, for every $x\in B$ and for every $n$,  
\begin{align*}
    &x_n=0\implies (G(x))_n=0,\\
    &x_n>0\implies (G(x))_n\geq 0,\\
    &x_n<0\implies (G(x))_n\leq 0.\\    
\end{align*}

We next show that $G$ fixes the extreme points of $B$. Recall that $\ext B=\{x\in B : x_n=\pm 1\ \text{for every $n$}\}$.

Fix $x\in \ext B$. By \cite[Theorem 2.3]{MR3534519}, $y\coloneqq G^{-1}(x)\in\ext B$. Thus $y_n=\pm 1$ for every $n\in\mathbb N$. Since $G$ preserves signs, if $y_n=1$, then $x_n=(G(y))_n\geq 0$, and hence $x_n=1$. If $y_n=-1$, then $x_n=(G(y))_n\leq 0$, and hence $x_n=-1$. Thus $x=y$ and $G(x)=x$.

It follows that for every $x\in\ext B$, $F(x)=A(G(x))=Ax\in\ext B$. Hence $F(\ext B)\subset \ext B$. By \cite[Corollary~5.3]{MR4633724}, $F$ is an isometry. This proves Theorem~\ref{thm: ell-infty-plastic}.

\section{A remark on \texorpdfstring{$\ell_\infty(\Gamma)$}{ell-infty(Gamma)}}
Let $\Gamma$ be a non-empty set and let $B$ denote the closed unit ball of the real Banach space $\ell_\infty(\Gamma)$.

\begin{proposition}
    Let $F\colon B\to B$ be a non-expansive bijection. Then 
    \[
    F(\ext B)\subset \ext B.
    \]
\end{proposition}

\begin{proof}
    The coordinate part of the proof of Theorem~\ref{thm: ell-infty-plastic} goes through verbatim with $\mathbb N$ replaced by $\Gamma$. More precisely, the analogue of Lemma~\ref{lem:2} says that if $0\neq a,b\in B$ and $B\subset B(a,1)\cup B(b,1)$, then $a$ and $b$ are non-zero opposite multiples of the same coordinate vector $e_\gamma$ for some $\gamma\in \Gamma$. Hence, there is an injective map $\sigma\colon \Gamma\to\Gamma$ and signs $\theta_\gamma\in\{-1,1\}$ such that, for every $x\in B$, 
    \[\supp(F(x))\subset \sigma(\supp x)\]
    together with the corresponding sign preservation. The same argument as above shows that $\sigma$ is surjective.

    Composing with the inverse signed permutation, we get a non-expansive bijection $G\colon B\to B$ which preserves coordinate signs. Since $\ext B=\{-1,1\}^\Gamma$, the general extreme point preimage result \cite[Theorem 2.3]{MR3534519} gives $G^{-1}(x)\in\ext B$ for every $x\in \ext B$. The sign-preservation argument from the proof of our main theorem then shows that $G$ fixes every extreme point. Therefore, $F$ sends every extreme point to an extreme point.
\end{proof}


\bibliographystyle{amsplain}
\bibliography{bibliography}
\end{document}